\documentclass[leqno, 12pt]{article}

\makeindex

\usepackage{amsmath, amssymb, amsthm, bm}

\usepackage[all]{xy}

\usepackage{mathtools}

\usepackage{hyperref}

\newcommand{\aut}{\mathrm{Aut}}

\renewcommand{\mod}{{\;\rm mod}}

\newcommand{\Ima}{{\rm Im}}

\newcommand{\cO}{ {\mathcal{O}}}

\newcommand{\cU}{{\mathcal{U}}}

\newcommand{\PP}{{\mathbb P}}

\newcommand{\ra}{\rightarrow}

\newcommand{\be}{\begin{equation}}
\newcommand{\ee}{\end{equation}}

\newcommand{\hone}{H^1(X,\cO_X)}

\newcommand{\hzero}{H^0(X, \Omega_X)}
\newcommand{\derhamhone}{H^1_\mathrm{dR}(X/k)}

\newtheorem{thm}{Theorem}[section]
\newtheorem*{thm*}{Theorem}
\newtheorem{prop}[thm]{Proposition}

\newtheorem{lem}[thm]{Lemma}
\newtheorem{rmk}[thm]{Remark}

\newtheorem{example}[thm]{Example}

\begin{document}

\begin{center}
{\bf \Large On the de-Rham Cohomology of}\\
\vspace*{0.3em} {\bf \Large Hyperelliptic Curves}\\
\vspace*{1em}
{\sc Bernhard K\"ock} and {\sc Joseph Tait}\\
\vspace*{2em}
\begin{quote}
{\bf Abstract.} {\footnotesize For any hyperelliptic curve $X$, we give an explicit basis of the first de-Rham cohomology of $X$ in terms of \v{C}ech cohomology. We use this to produce a family of curves in characteristic~$p>2$ for which the Hodge-de-Rham short exact sequence does not split equivariantly; this generalises a result of Hortsch. Further, we use our basis to show that the hyperelliptic involution acts on the first de-Rham cohomology by multiplication by $-1$, i.e., acts as the identity when $p=2$.

{\bf MSC-class:} 14F40, 14G17, 14H37.

}\end{quote}
\end{center}

\section*{Introduction}

Recall that the de-Rham cohomology $H^*_\mathrm{dR}(X/k)$ of a smooth projective curve $X$ over an algebraically closed field $k$ is defined as the hypercohomology of the de-Rham complex
\[\cO_X \xrightarrow{d} \Omega_X\]
where $d$ denotes the usual differential map $f \mapsto df$. In particular, we have a long exact sequence relating $H^*_\mathrm{dR}(X/k)$ to ordinary cohomology of the structure sheaf $\cO_X$ and of the sheaf $\Omega_X$ of differentials on~$X$. The very general and famous fact that the Hodge-de-Rham spectral sequence degenerates at $E_1$ (e.g., see \cite{wedhorn}) means for our curve $X$ that the following main part of that long sequence is a short exact sequence, see Propositon~\ref{propshortexactsequence}:
\begin{equation*}
0 \rightarrow H^0(X, \Omega_X) \rightarrow H^1_\mathrm{dR}(X/k) \rightarrow H^1(X, \cO_X) \rightarrow 0.
\end{equation*}
We call this sequence the Hodge-de-Rham short exact sequence. In particular, the vector space $H^1_\mathrm{dR}(X/k)$ is the direct sum of the vector spaces $H^0(X, \Omega_X)$ and $H^1(X, \cO_X)$ over $k$.

We now assume furthermore that a finite group $G$ acts on our curve $X$. If $p:= \mathrm{char}(k)$ does not divide the order of $G$, Maschke's Theorem implies that the Hodge-de-Rham short exact sequence also splits as a sequence of modules over the group ring $k[G]$.

However, the latter fact fails to be true in general when $p>0$ does divide $\mathrm{ord}(G)$. A counterexample has been constructed in the recent paper \cite{canonicalrepresentation} by Hortsch.
The main goal of this paper is to generalise that counterexample. More precisely, we will prove the following theorem, see Theorem~\ref{theoremsplittingtheorem} and Example~\ref{examplesofsutomorpohism}.\\

{\bf Theorem}. {\em Let $p \ge 3$ and  let $q(z) \in k[z]$ be a monic polynomial of odd degree without repeated roots. Let $X$ denote the hyperelliptic curve over $k$ defined by the equation $y^2=q(x^p-x)$ and let $G$ denote the subgroup of $\mathrm{Aut}(X)$ generated by the automorphism $\tau$ given by $(x,y) \mapsto (x+1, y)$. Then the Hodge-de-Rham short exact sequence does not split as a sequence of $k[G]$-modules.} \\

We remark that the hyperelliptic curves considered in this theorem are exactly those hyperelliptic curves $y^2=f(x)$ which allow an automorphism that maps $x$ to $x+1$ and for which $f(x)$ is of odd degree, see Example~\ref{examplesofsutomorpohism} and Proposition~\ref{lemmatauactsbyplusminusoneony}.

When $q(z)= z$, the theorem above becomes the main theorem of \cite{canonicalrepresentation}. Beyond \cite{canonicalrepresentation}, our theorem shows (see Remark~\ref{RemarksMainTheorem}) that, for every algebraically closed field $k$ of characteristic $p\ge 3$, there exist infinitely many $g \ge 2$ and hyperelliptic curves $X$ over $k$ of genus $g$ for which the Hodge-de-Rham short exact sequence does not split equivariantly. It also shows that, for every $g \ge 2$, there exists a prime $p\ge 3$ and  hyperelliptic curves in characteristic $p$ of genus $g$ for which the Hodge-de-Rham short exact sequence does not split equivariantly.

In Example~\ref{ExampleBranchPoint} and Remark~\ref{RemarkBranchPoint} we show, using the modular curve $X_0(22)$ for $p=3$, that, without assuming the degree of $q(x)$ to be odd, this theorem may be false.

To prove our main theorem, we follow the same broad strategy as in \cite{canonicalrepresentation}: we give an explicit basis of $H^1_\mathrm{dR}(X/k)$ in terms of \v{C}ech cohomology (in fact for an arbitrary hyperelliptic curve $X$), see Theorem~\ref{theorembasisofderham}, and study the action of $\tau$ on that basis. The actual computations towards the end however do not generalise those in \cite{canonicalrepresentation}, see Remark~\ref{mistake} and Remark~\ref{mistake2}.

We provide a basis of $H^1_\mathrm{dR}(X/k)$ for any hyperelliptic curve $X$ also when $p=2$ and use this to show that the hyperelliptic involution acts trivially on $H^1_\mathrm{dR}(X/k)$ when $p=2$. In fact, the hyperelliptic involution acts on $H^1_\mathrm{dR}(X/k)$ by multiplication by $-1$ for all $p$, see Theorem~\ref{thmactionhyperelliptic}.

If $p=2$, Elkin and Pries construct a subtler basis of $H^1_\mathrm{dR}(X/k)$ in \cite{pries} which is suitable to study the action of Frobenius and Verschiebung and, finally, to determine the Ekedahl-Oort type.

{\bf Acknowledgements.} The authors would like to thank the referees for carefully reading the paper and for their elaborate and helpful comments.

\section{Preliminaries}

In this section, we introduce assumptions and notations used throughout this paper and collect and prove some auxiliary results.

We assume that $k$ is an algebraically closed field of characteristic $p \ge 0$ and that $X$ is a hyperelliptic curve over $k$ of genus $g \geq 2$.
We recall that a curve (always assumed to be smooth, projective and irreducible in this paper) is hyperelliptic if there exists a finite, separable morphism of degree two from the curve to $\mathbb P_k^1$. \index{Hyperelliptic curve}
We fix such a map
\[\pi \colon X \rightarrow \mathbb P_k^1,\]
which is unique up to automorphisms of $X$ and of $\mathbb P_k^1$ (see \cite[Remark~7.4.30]{liu}). Let $K(X)/K(\PP^1_k)= k(x)$ denote the extension of function fields corresponding to~$\pi$.
According to \cite[Proposition~7.4.24 and Remark~7.4.25]{liu},
we may and will furthermore assume the following concrete description of $K(X)$.

If $p \not=2$, then $K(X)=k(x,y)$ where $y$ satisfies
\begin{equation}\label{Weierstrassequationpnot2}
y^2 = f(x)
\end{equation}
for some monic polynomial $f(x) \in k[x]$ which has no repeated roots; moreover, $f(x)$ is of degree $2g +1$ if $\infty \in \PP^1_k$ is a branch point of $\pi$ and of degree $2g+2$ otherwise. The branch points of~$\pi$ are then the roots of $f(x)$, together with $\infty \in \PP^1_k$ if $\mathrm{deg}(f(x)) = 2g+1$.

If $p=2$, then $K(X) = k(x,y)$ where $y$ satisfies
\begin{equation}\label{WeierstrassEquationpis2}
y^2 - h(x)y = f(x)
\end{equation}
for some polynomials $h(x), f(x) \in k[x]$ such that $h'(x)^2 f(x) + f'(x)^2$ and $h(x)$ have no common roots in $k$; moreover, we have $d:=\mathrm{deg}(h(x)) \le g+1$, with equality if and only if $\infty$ is not a branch point of $\pi$.
The branch points of~$\pi$ are the roots of $h(x)$,
together with $\infty\in \PP^1_k$ if $d<g+1$.

The following estimate for the order of $y$ above $\infty$ is true for both $p \not=2$ and $p=2$.

\begin{lem}\label{order of y at infty}
Let $P \in \pi^{-1}(\infty)$. Then we have:
\[\mathrm{ord}_P(y) \ge \left\{\begin{aligned}
&-(g+1)&& \quad \textrm{ if $\pi$ is unramified at $P$}\\
&-2(g+1) && \quad \textrm{ if $\pi$ is ramified at $P$}.
\end{aligned}\right.\]
\end{lem}

\begin{proof}
  This is \cite[Inequality~(5.2)]{faithful}.
\end{proof}

\begin{lem}\label{bais of global holomorphic differentials}
If $p \not=2$, let $\omega := \frac{dx}{y}$ and, if $p=2$, let $\omega:= \frac{dx}{h(x)}$. Then the differentials $\omega, x\omega, \ldots, x^{g-1}\omega$ form a basis of the $k$-vector space $H^0(X, \Omega_X)$ of global holomorphic differentials on $X$.
\end{lem}

\begin{proof}
  This is \cite[Proposition~7.4.26]{liu}.
\end{proof}

\begin{rmk}
{\em A different basis of $H^0(X,\Omega_X)$ is given in \cite[Lemma~5]{madden}. The action of the Cartier operator on $H^0(X,\Omega_X)$ is studied in \cite{subrao} and \cite{Yui}.}
\end{rmk}
\begin{lem}\label{order of dx above infty}
Let $p=2$ and let $P \in \pi^{-1}(\infty)$. Then we have:
\begin{equation}
\mathrm{ord}_{P}(dx) =
\left\{\begin{aligned}
 &-2 && \quad \textrm{ if $\pi$ is unramified at $P$}\\
&2(g-1-d)& & \quad \textrm{ if $\pi$ is ramified at $P$}.
\end{aligned} \right.
\end{equation}
\end{lem}

\begin{proof}
  By the Riemann-Hurwitz formula \cite[Theorem~3.4.6]{stichtenoth} we have
  \[\mathrm{ord}_P(dx) = e_P \cdot \mathrm{ord}_\infty(dx) + \delta_P\]
  where $e_P$ denotes the ramification index of $\pi$ at $P$ and $\delta_P$ denotes the order of the ramification divisor of $\pi$ at $P$. It is easy to see that $\mathrm{ord}_\infty(dx) = -2$. Therefore $\mathrm{ord}_P(dx) = -2$ if $\pi$ is unramified at $P$. On the other hand, if $\pi$ is ramified at $P$, we have $\delta_P = 2(g+1-d)$ by \cite[Equation~(5.3)]{faithful} and hence
  \[\mathrm{ord}_P(dx) = 2 \cdot (-2) + 2(g+1-d) = 2(g-1-d),\]
  as claimed.
\end{proof}

We define $U_a = X \backslash \pi^{-1}(a)$ for any $a \in \mathbb P_k^1$ and let ${\cal U}$ be the affine cover of $X$ formed by $U_0$ and $U_\infty$.
Given any sheaf $\cal F$ on $X$ we have the \v{C}ech differential $\check{d}\colon {\cal F}(U_0) \times {\cal F} (U_\infty) \rightarrow {\cal F}(U_0 \cap U_\infty)$, defined by $(f_0,f_\infty) \mapsto f_0|_{U_0 \cap U_\infty} - f_\infty|_{U_0 \cap U_\infty}$.
In general we will suppress the notation denoting the restriction map.
The first cohomology group $\frac{\cO_X(U_0 \cap U_\infty)}{\Ima(\check{d})}$ of the cochain complex
    \begin{equation*}
    0 \rightarrow \cO_X(U_0)\times \cO_X(U_\infty) \xrightarrow{\check{d}} \cO_X(U_0 \cap U_\infty) \rightarrow 0.
    \end{equation*}
is the first \v{C}ech cohomology group $\check{H}^1({\cU},\cO_X)$. By Leray's theorem \cite[Theorem~5.2.12]{liu} and Serre's affineness criterion \cite[Theorem~5.2.23]{liu} we therefore have
    \begin{equation}\label{equationcechhoneisomorphismfunctions}
    \hone \cong
    \frac{\cO_X(U_0 \cap U_\infty)}{\{f_0 - f_\infty \mid f_0 \in \cO_X(U_0), \, f_\infty \in \cO_X(U_\infty) \}}.
    \end{equation}
When describing elements of $\hone$ using this isomorphisms we will denote the residue class of $f \in \cO_X(U_0 \cap U_\infty)$   by $[f]$.

\begin{prop}\label{theorembasisofhone}
    The elements $\frac{y}{x}, \ldots, \frac{y}{x^g} \in K(X)$ are regular on $U_0 \cap U_\infty$, and their residue classes $\left [ \frac{y}{x} \right ],  \ldots, \left [ \frac{y}{x^g} \right]$ form a basis of $\hone$.
    \end{prop}

\begin{proof}
By \cite[Proposition~7.4.24(b)]{liu}, we may identify $\cO_X(U_\infty)$ with the $k$-algebra $k[x,y]$ defined by the relation given in (\ref{Weierstrassequationpnot2}) or (\ref{WeierstrassEquationpis2}). Then $\cO_X(U_0 \cap U_\infty)$ is $k[x^{\pm 1},y]$. As the relations in (\ref{Weierstrassequationpnot2}) and (\ref{WeierstrassEquationpis2}) are quadratic in $y$, the elements $\ldots, \frac{1}{x^2}, \frac{1}{x}, 1, x, x^2, \ldots$ and $\ldots, \frac{y}{x^2}, \frac{y}{x}, y, xy, x^2y, \ldots$ form a $k$-basis of $k[x^{\pm 1},y]$. The elements $1, x, x^2, \ldots$ and $y, xy, x^2y, \ldots$ obviously form a basis of the image of $\cO_X(U_\infty)$ in $\cO_X(U_0\cap U_\infty)$. By \cite[Proposition~7.4.24(b)]{liu}, the image of $\cO_X(U_0)$ in $\cO_X(U_0\cap U_\infty)$ consists of elements of the form $g\left(\frac{1}{x}, \frac{y}{x^{g+1}}\right)$ where $g \in k[s,t]$. Hence, the elements $\ldots, \frac{1}{x^2}, \frac{1}{x}, 1$ and $\ldots \frac{y}{x^{g+3}}, \frac{y}{x^{g+2}}, \frac{y}{x^{g+1}}$ form a basis of that image. We conclude that the residue  classes $\left [ \frac{y}{x} \right ],  \ldots, \left [ \frac{y}{x^g} \right]$ form a basis of $\hone$, as was to be shown.
\end{proof}

\begin{rmk} {\em Let $\omega_j := \frac {x^{j-1}}{y} {dx}$  when $p \not=2$ and let $\omega_j =\frac{ x^{j-1}}{h(x)}{dx}$ when $p=2$. Then, by Lemma~\ref{bais of global holomorphic differentials}, the elements $\omega_j$, $j=1, \ldots, g$, form a $k$-basis of~$H^0(X,\Omega_X)$. Let $\langle\hspace{0.7em}, \hspace{0.7em}\rangle:H^0(X, \Omega_X) \times H^1(X, \cO_X) \rightarrow k$ denote the Serre duality pairing. Then $\langle \omega_j,\left[ \frac{y}{x^i} \right] \rangle $ vanishes if $j \not= i$ and is non-zero if $j=i$, see the proof of \cite[Theorem~4.2.1]{taitthesis}. In other words, up to multiplication by scalars, the basis $\left[ \frac{y}{x^i} \right]$, $i=1, \ldots, g$, of $\hone$, given in Proposition~\ref{theorembasisofhone}, is dual to the basis $\omega_j$, $j=1, \ldots, g$, with respect to Serre duality}.
\end{rmk}

\begin{rmk}
{\em Different bases of $H^1(X,\cO_X)$ are described in \cite[Lemma~6]{sullivan} and \cite[Lemma~6]{madden}. The action of Frobenius on $H^1(X,\cO_X)$ is studied in \cite{bouw}.
}
\end{rmk}

\section{Bases of $H^1_\textrm{dR}(X/k)$}

The object of this section is to give an explicit $k$-basis for the first de-Rham cohomology group $H^1_\mathrm{dR}(X/k)$ using \v{C}ech cohomology. If $p \not= 2$, we will moreover refine our result when another open subset is added to our standard open cover of $X$.

The algebraic de-Rham cohomology of $X$  is defined to be the hypercohomology of the de-Rham complex
\begin{equation}\label{equationderhamcomplex}
 0 \rightarrow \cO_X \xrightarrow{d} \Omega_X \rightarrow 0
\end{equation}
where $d$ denotes the usual differential map $f \mapsto df$.
We use the cover $\cal U$ and the \v{C}ech differentials defined in the previous section to obtain the \v{C}ech bicomplex of \eqref{equationderhamcomplex}:
    \begin{equation}\label{equationcechbicomplex}
    \begin{split}
    \xymatrix{ & 0 \ar[d] & 0 \ar[d] & \\
    0 \ar[r] & \cO_X(U_0) \times \cO_X(U_\infty) \ar[d] \ar[r] & \Omega_X(U_0) \times \Omega_X(U_\infty) \ar[d] \ar[r] & 0 \\
    0 \ar[r] & \cO_X(U_0\cap U_\infty) \ar[d] \ar[r] & \Omega_X(U_0 \cap U_\infty) \ar[r] \ar[d] & 0 \\
    & 0 & 0 &}
    \end{split}
    \end{equation}
By a generalisation of Leray's theorem \cite[Corollaire~12.4.7]{EGA0III} and Serre's affineness criterion \cite[Theorem~5.2.23]{liu}, the first de-Rham cohomology of~$X$ is isomorphic to the first cohomology of the total complex of \eqref{equationcechbicomplex}.
Thus, $\derhamhone$ is isomorphic to the quotient of the space
    \begin{equation}\label{equationderhamspace}
    \begin{split}
    \big\{(\omega_0, \omega_\infty, f_{0\infty}) \in \Omega_X(U_0)\times \Omega_X&(U_\infty) \times \cO_X(U_0 \cap U_\infty)\mid\\
    &df_{0\infty} = \omega_0|_{U_0\cap U_\infty} - \omega_\infty|_{U_0\cap U_\infty} \big\}
    \end{split}
    \end{equation}
by the subspace
    \begin{equation}\label{equationderhamquotient}
    \left\{  (df_0, df_\infty, f_0|_{U_0 \cap U_\infty} -f_\infty|_{U_0 \cap U_\infty} )|f_0 \in \cO_X(U_0), f_\infty \in \cO_X(U_\infty) \right\}.
    \end{equation}
Via this representation of $\derhamhone$ and the isomorphism \eqref{equationcechhoneisomorphismfunctions} we obtain the canonical maps
    \begin{equation}\label{equationinjectionofhzerointoderham}
    i \colon \hzero \ra \derhamhone, \qquad \omega \mapsto [(\omega|_{U_0}, \omega|_{U_\infty}, 0)]
    \end{equation}
and
    \begin{equation}\label{equationsurjectionofderhamontohone}
    p \colon \derhamhone \ra \hone, \qquad [(\omega_0, \omega_\infty, f_{0 \infty})] \mapsto [f_{0 \infty}].
    \end{equation}

The following proposition is equivalent to the more familiar and fancier sounding statement that the Hodge-de-Rham spectral sequence for $X$ degenerates at $E_1$ (see \cite{wedhorn}). This is in fact true for every smooth, proper curve~$X$ over $k$, see example (2) in section~(1.5) of \cite{wedhorn}.
\begin{prop}\label{propshortexactsequence}
    The following sequence is exact:
        \begin{equation}\label{equationderhamcohomologyshortexactseqeunce}
        0 \ra H^0(X,\Omega_X) \xrightarrow{i} \derhamhone \xrightarrow{p} H^1(X,\cO_X) \ra 0.
        \end{equation}
    \end{prop}

We will call the sequence~\eqref{equationderhamcohomologyshortexactseqeunce} the {\em Hodge-de-Rham short exact sequence}.\\

An elementary proof of Proposition~\ref{propshortexactsequence} (that works for every smooth projective curve) can be found in \cite[Proposition~4.1.2]{taitthesis}; the main ingredient there is just the fact that the residue of differentials of the form~$df$ vanishes at every point of $X$ and that hence the obvious composition $H^1(X,\cO_X) \rightarrow H^1(X, \Omega_X) \overset{\sim}{\rightarrow} k$ is the zero map. For a hyperelliptic curve $X$, the surjectivity of~$p$ will also be verified in the proof of Theorem~\ref{theorembasisofderham} below.

\hspace*{1em}

In order to state a basis of $\derhamhone$,
we now define certain polynomials. To this end, we introduce the notations $f^{\le m}(x):= a_0 + \ldots +a_mx^m$ and $f^{> m}(x):=a_{m+1}x^{m+1}+ \ldots + a_nx^n$ for any polynomial $f(x):= a_0 + \ldots + a_n x^n \in k[x]$ and any $m \ge 0$. Let $1 \leq i \leq g$.

When $p\neq 2$ we define
    \[
    s_i(x) := xf'(x) - 2if(x) \in k[x]
    \]
and put $\psi_i(x):= s_i^{\le i}(x)$ and $\phi_i(x):= s_i^{>i}(x)$ so that $s_i(x) = \psi_i(x) + \phi_i(x)$.

When $p = 2$ we define
    \begin{equation*}\label{equationSi}
    s_i(x,y) := xf'(x) + (xh'(x) + ih(x))y\in k[x]\oplus k[x]y \subseteq k(x,y)
    \end{equation*}
(where $k[x]\oplus k[x]y $ denotes the $k[x]$-module generated by $1$ and $y$)
and put $\psi_i(x,y) := s_i^{\le i}(x,y)$ and $\phi_i(x,y) := s_i^{> i}(x,y)$ where now the operations~$\le i$ and~$> i$ are applied to both the coefficients  $xf'(x)$ and $xh'(x) + ih(x)$. Again we have $s_i(x,y) = \psi_i(x,y)+ \phi_i(x,y)$.

We now give a basis of $\derhamhone$ in terms of the polynomials just introduced and using the presentation of $\derhamhone$ developed above.

    \begin{thm}\label{theorembasisofderham}
    If $p \neq 2$, the residue classes
        \begin{equation}\label{equationhonebasiselementofderhampnot2}
        \gamma_i:= \left[ \left( \frac{\psi_i(x)}{2x^{i+1}y} dx, \frac{-\phi_i(x)}{2x^{i+1}y} dx, \frac{y}{x^i} \right)\right] , \quad i=1, \ldots ,g,
        \end{equation}
    along with the residue classes
        \begin{equation}\label{equationhzerobasiselementofderhampnot2}
         \lambda_i:= \left[ \left( \frac{x^{i}}{y} dx , \frac{x^{i}}{y} dx, 0 \right)\right] , \quad i = 0,\ldots ,g-1,
        \end{equation}
    form a $k$-basis of $\derhamhone$.

    On the other hand, if $p=2$, the residue classes
        \begin{equation}\label{equationhonebasiselementofderhampis2}
        \gamma_i := \left[ \left( \frac{\psi_i(x,y)}{x^{i+1}h(x)} dx, \frac{\phi_i(x,y)}{x^{i+1}h(x)} dx, \frac{y}{x^i} \right)\right], \quad i =1, \ldots , g,
        \end{equation}
    together with the residue classes
        \begin{equation}\label{equationhzerobasiselementofderhampis2}
        \lambda_i := \left[ \left( \frac{x^{i}}{h(x)} dx, \frac{x^{i}}{h(x)} dx, 0 \right)\right], \quad i=0, \ldots, g-1,
        \end{equation}
    form a $k$-basis of $\derhamhone$.
    \end{thm}

\begin{rmk} \hspace*{1em}\\
{\em (a) If $p \not =2$ and $f(x) = x^p -x$, an easy calculation shows that the basis elements given above are the same as those given in Theorem~3.1 of \cite{canonicalrepresentation}.\\
(b) If $p=2$, another basis of $H^1_\mathrm{dR}(X/k)$ is given in \cite[Section~4]{pries}. }
\end{rmk}

\begin{proof} The elements in (\ref{equationhzerobasiselementofderhampnot2}) and (\ref{equationhzerobasiselementofderhampis2}) are the images under the map $i$ of the differentials $\frac{x^i}{y}dx$, $i=0, \ldots, g-1$, and $\frac{x^i}{h(x)}dx$, $i=0, \ldots, g-1$, respectively, which form a basis of $H^0(X, \Omega_X)$ by Lemma~\ref{bais of global holomorphic differentials}. Furthermore, provided the elements in (\ref{equationhonebasiselementofderhampnot2}) and (\ref{equationhonebasiselementofderhampis2}) are well-defined elements of $\derhamhone$, these elements are mapped to the elements $[\frac{y}{x^i}]$, $i=1, \ldots, g$, under $p$, which form a basis of $\hone$ by Proposition~\ref{theorembasisofhone}. By Proposition~\ref{propshortexactsequence}, it therefore suffices to check that the elements in (\ref{equationhonebasiselementofderhampnot2}) and (\ref{equationhonebasiselementofderhampis2}) are well-defined elements of $\derhamhone$.

We first check the equality in ~(\ref{equationderhamspace}). When $p \not=2$, this  is verified  as follows:
  \begin{align*}
        &\left(  \frac{\psi_i(x)}{2x^{i+1}y}  - \frac{-\phi_i(x)}{2x^{i+1}y} \right) dx  =  \frac{s_i(x)}{2x^{i+1}y} dx \\
        &\hspace*{3em}  =  \frac{xf'(x) - {2if(x)}}{2x^{i+1}y} dx
        =  \frac{x^i}{2y} \left( \frac{f'(x)}{x^{2i}} -\frac{2if(x)}{x^{2i+1}} \right)  dx\\
         & \hspace*{3em} =  \frac{x^i}{2y}d\left(\frac{f(x)}{x^{2i}}\right)
         =  \frac{x^i}{2y} d\left(\left(\frac{y}{x^{i}}\right)^2\right)
        =  d\left(\frac{y}{x^{i}}\right).
        \end{align*}
When $p=2$, we obtain
\[h'(x)ydx+h(x)dy=f'(x)dx\]
by differentiating equation~(\ref{WeierstrassEquationpis2}) and then verify the equality in (\ref{equationderhamspace}) as follows (note that we replace all minus signs with plus signs):
        \begin{align*}
        &\left(  \frac{ \psi_i(x,y)}{x^{i+1}h(x)}  +  \frac{\phi_i(x,y)}{x^{i+1}h(x)}  \right) dx =  \frac{s_i(x,y)}{x^{i+1}h(x)}dx \\
        &\hspace*{3em} =  \left( \frac{f'(x)}{x^ih(x)} + \frac{h'(x)y}{x^ih(x)} + \frac{iy}{x^{i+1}} \right) dx
        \\
        & \hspace*{3em} =  \frac{dy}{x^{i}} + \frac{iy}{x^{i+1}} dx
        =  d\left( \frac{y}{x^{i}}\right).
        \end{align*}

It remains to prove that the first two entries of the triples in (\ref{equationhonebasiselementofderhampnot2}) and~(\ref{equationhonebasiselementofderhampis2}) are regular differentials on $U_0$ and $U_\infty$, respectively.

We first consider the case $p \not= 2$. As $\frac{dx}{y}$ is a regular differential on $X= U_0 \cup U_\infty$ by Lemma~\ref{bais of global holomorphic differentials}, it suffices to observe that each of the functions $\frac{\psi_i(x)}{x^{i+1}}$, $i=1, \ldots, g$, is regular on $U_0$ (in fact has a zero at $\infty$) and that each of the functions $\frac{\phi_i(x)}{x^{i+1}}$, $i=1, \ldots, g$, is regular on $U_\infty$.

We now turn to the case $p=2$. As above, we know from Lemma~\ref{bais of global holomorphic differentials} that $\frac{dx}{h(x)}$ is regular on $X=U_0 \cup U_\infty$. Furthermore, for every $i \in \{1, \ldots, g\}$, the function $\frac{\phi_i(x,y)}{x^{i+1}}$ is regular on $U_\infty$ since $y$ is regular on $U_\infty$ and since, by definition of $\phi_i(x,y)$, the $k[x]$-coefficients of $1$ and $y$ in $\phi_i(x,y)$ are divisible by $x^{i+1}$. Hence $\frac{\phi_i(x,y)}{x^{i+1}h(x)} dx$ is regular on $U_\infty$, as was to be shown. It remains to show that $ \frac{\psi_i(x,y)}{x^{i+1}h(x)} dx$ is regular on $U_0$. As $\frac{\psi_i(x,y)}{x^{i+1}}$ and $\frac{dx}{h(x)}$ are regular on $U_0 \cap U_\infty$, this amounts to showing that $ \frac{\psi_i(x,y)}{x^{i+1}h(x)} dx$ is regular above $\infty$.

We first consider the case when $\infty$ is not a branch point of $\pi$.
By Lemma~\ref{order of dx above infty}, the differential~$dx$ has a pole of order $2$ at each of the two points $P_\infty$, $P'_\infty \in X$ above $\infty$. Furthermore, the $k[x]$-coefficient of $1$ in $\psi_i(x,y)$ has a pole at $P_\infty$ and $P'_\infty$ of order at most $i$ and the $k[x]$-coefficient of $y$ has a pole at $P_\infty$ and $P'_\infty$ of order at most $i-1$ since the coefficient of $x^i$ in $xh'(x) + ih(x)$ is zero (remember $\textrm{char}(k)=2$). Moreover, $y$ has a pole at $P_\infty$ and $P'_\infty$ of order at most $g+1$ by Lemma~\ref{order of y at infty}. Finally, $\frac{1}{h(x)}$ has a zero at $P_\infty$ and $P'_\infty$ of order $d=\textrm{deg}(h(x)) = g+1$. Putting all this together we obtain
\begin{align*}
\mathrm{ord}&_{P} \left( \frac{\psi_i(x,y)}{x^{i+1} h(x)}dx  \right)\\
&= \mathrm{ord}_{P}\left(\psi_i(x,y)\right) + \mathrm{ord}_{P}\left(\frac{1}{x^{i+1}} \right)+ \mathrm{ord}_{P}\left(\frac{1}{h(x)}\right) + \mathrm{ord}_{P}(dx)\\
&\ge \min\{-i, -(i-1)-(g+1)\} + (i+1) + (g+1) -2 =0
\end{align*}
for $P \in \{P_\infty, P'_\infty\}$, which
shows that $\frac{\psi_i(x,y)}{x^{i+1}h(x)} dx$ is regular at $P_\infty$ and $P'_\infty$.

We finally assume that $\infty$ is a branch point of $\pi$ and prove that $\frac{\psi_i(x,y)}{x^{i+1}h(x)} dx$ is regular at the unique point $P_\infty \in X$ above $\infty$. By Lemma~\ref{order of dx above infty}, the order of the differential $dx$ at $P_\infty$ is $2(g-1-d)$ where $d= \mathrm{deg}(h(x))$.
For similar reasons as above, the $k[x]$-coefficients of $1$ and $y$ in $\psi_i(x,y)$ have a pole at~$P_\infty$ of order at most $2i$ and $2(i-1)$, respectively, and $\frac{1}{h(x)}$ has a zero at $P_\infty$ of order $2d$. Finally, $y$ has a pole at $P_\infty$ of order at most $2(g+1)$ by Lemma~\ref{order of y at infty}. Putting all this together we obtain
\begin{align*}
\mathrm{ord}&_{P_\infty} \left( \frac{\psi_i(x,y)}{x^{i+1} h(x)}dx  \right)\\
&= \mathrm{ord}_{P_\infty}\left(\psi_i(x,y)\right) + \mathrm{ord}_{P_\infty}\left(\frac{1}{x^{i+1}} \right)+ \mathrm{ord}_{P_\infty}\left(\frac{1}{h(x)}\right) + \mathrm{ord}_{P_\infty}(dx)\\
&\ge \min\{-2i, -2(i-1)-2(g+1)\} + 2(i+1) + 2d + 2(g-1-d) =0,
\end{align*}
which shows that $\frac{\psi_i(x,y)}{x^{i+1}h(x)} dx$ is regular at $P_\infty$.
\end{proof}

In the proofs in the next section, we will need a refined description of the basis elements given in \eqref{equationhonebasiselementofderhampnot2} when another open subset is added to our standard cover $\cU= \{U_0, U_\infty\}$. To this end, we now fix $a \in \PP^1_k \backslash \{0, \infty\}$ and define the covers $\cU':=\{U_a, U_\infty\}$ and $\cU'' := \{U_0, U_a, U_\infty\}$ of $X$.
Similarly to~(\ref{equationderhamspace}) and (\ref{equationderhamquotient}), the first de-Rham cohmology group $\derhamhone$ is then isomorphic to the $k$-vector space
    \begin{multline}\label{equationsixtupleconditions}
    \left\{ (\omega_0, \omega_a, \omega_\infty , f_{0a}, f_{0 \infty},f_{a \infty}) \in  \right.\\
    \Omega_X(U_0) \times \Omega_X(U_a)\times \Omega_X(U_\infty)
    \times \cO_X(U_0 \cap U_a) \times \cO_X(U_0 \cap U_\infty) \times \cO_X(U_a \cap U_\infty) \mid \\
    \left. f_{0a} - f_{0\infty} + f_{a \infty} = 0,\;
    df_{0a} = \omega_0 - \omega_a, \; df_{0\infty} = \omega_0 - \omega_\infty, \; df_{a\infty} = \omega_a - \omega_\infty\right\}
    \end{multline}
quotiented by the subspace
    \begin{multline}\label{equationsixtuplequotient}
    \left\{( df_0, df_a, df_\infty, f_0- f_a, f_0 - f_\infty, f_a - f_\infty ) | \right. \\
    \left.f_0 \in \cO_X(U_0), f_a \in \cO_X(U_a ), f_\infty \in \cO_X(U_\infty)\right\}.
    \end{multline}

We use the notations $\check{H}^1_\mathrm{dR}(\cU)$ and $\check{H}^1_\mathrm{dR}(\cU'')$ for the representations of $\derhamhone$ introduced in (\ref{equationderhamspace}), (\ref{equationderhamquotient}) and (\ref{equationsixtupleconditions}), (\ref{equationsixtuplequotient}), respectively.
The canonical isomorphism $\rho\colon \check{H}^1_\mathrm{dR}(\cU'') \ra \check{H}^1_\mathrm{dR}(\cU)$, is then induced by the projection
    \begin{equation}\label{equationdefinitionofrho}
    \rho \colon (\omega_0, \omega_a, \omega_\infty , f_{0a}, f_{0 \infty},f_{a \infty}) \mapsto (\omega_0, \omega_\infty , f_{0 \infty}).
    \end{equation}

When $p \not=2$, the next proposition explicitly describes the pre-image of the basis elements  $\gamma_i =  \left[ \left ( \frac{\psi_i(x)}{2x^{i+1}y}dx, \frac{-\phi_i(x)}{2x^{i+1}y}dx, \frac{y}{x^i}\right) \right],  i=1,\ldots ,g$, of $\derhamhone$ under~$\rho$. To this end, we define the polynomials
\[g(x) := (x-a)^g, \quad r_i(x) := g^{\le i -1}(x) \quad \textrm{and} \quad t_i(x) := g^{> i-1}(x)\]
in $k[x]$ for $1 \leq i \leq g$ so that $r_i(x) + t_i(x) = (x-a)^g$.

\begin{prop}\label{propbasisoftriplecoverderham}
Let $p \not= 2$. For $i \in \{1, \ldots, g\}$, let
\[\omega_{0i}:= \frac{\psi_i(x)}{2x^{i+1}y}dx, \quad \omega_{\infty i}:= \frac{-\phi_i(x)}{2x^{i+1}y}dx,\]
\[\omega_{ai}:= \frac{(\psi_i(x)t_i(x) - \phi_i(x)r_i(x))(x-a) - 2if(x)(-1)^{g-i}\binom{g}{i} a^{g-i+1}x^i}{2x^{i+1}(x-a)^{g+1}y}dx\]
and
\[f_{0ai}:= \frac{r_i(x)y}{x^i(x-a)^g}, \quad  f_{0 \infty i}:= \frac{y}{x^i}, \quad f_{a \infty i}:= \frac{t_i(x)y}{x^i(x-a)^g}.\]
Then we have:
\begin{equation}\label{preimage}
\rho^{-1}(\gamma_i) =[( \omega_{0 i}, \omega_{a i}, \omega_{\infty i}, f_{0 a i}, f_{0 \infty i}, f_{a \infty i} )].
\end{equation}
\end{prop}

\begin{rmk}\label{mistake}
{\em This description of $\rho^{-1}(\gamma_i)$ does not generalise the description given  in Lemma~3.3 of \cite{canonicalrepresentation} in case of the hyperelliptic curve $y^2 = x^p - x$. In fact, the proof of that lemma seems to contain various mistakes.}
\end{rmk}

\begin{proof}
We fix $i \in \{1, \ldots, g\}$. We obviously only need to show that the sextuple on the right-hand side of (\ref{preimage}) is a well-defined element of the space~(\ref{equationsixtupleconditions}).

From the proof of Theorem \ref{theorembasisofderham} { we already know that ${d(f_{0 \infty i}) = \omega_{0 i} - \omega_{\infty i}}$ and that ${f_{0 \infty i}, \omega_{0 i}}$ and ${\omega_{\infty i}}$ are regular on the appropriate open sets}.

Since $r_i(x)+t_i(x)=(x-a)^g$, we have
        \[
       {f_{0 a i} - f_{0 \infty i}+ f_{a \infty i}  = }\frac{r_i(x)y}{x^i(x-a)^g} - \frac{y}{x^i} + \frac{t_i(x)y}{x^i(x-a)^g}  = {0},
        \]
as desired.

The function $f_{0ai}$ is obviously regular above $\PP^1_k \backslash \{0, a, \infty\}$. We furthermore observe that $\mathrm{ord}_\infty\left(\frac{r_i(x)}{x^i (x-a)^g}\right) \ge -(i-1)  +i +g = g+1$ and that, by Lemma~\ref{order of y at infty}, the order of $y$ above $\infty$ is at least $-2(g+1)$ or at least $-(g+1)$ depending on whether $\infty$ is a branch point of $\pi$ or not. Thus, ${f_{0ai}}$ { is regular} above $\infty$ and hence { on ${U_0 \cap U_a}$}.

As above, the function $f_{a \infty i}$ is regular above $\PP^1_k \backslash \{0, a, \infty\}$. Furthermore, the functions $\frac{t_i(x)}{x^i}$, $y$ and $\frac{1}{(x-a)^g}$ are obviously regular above $0$. Therefore, ${f_{a \infty i}}$ { is regular} above $0$ as well and hence { on} ${U_a \cap U_\infty}$.

{ We next show that} ${df_{0 a i} = \omega_{0 i} - \omega_{a i}}$.
Using the product rule and the chain rule we obtain
        \begin{align*}
        & df_{0 a i}  = d \left( \frac{r_i(x)y}{x^i(x-a)^g} \right) \\
        & = \frac{r_i(x)}{x^i(x-a)^g}dy + d\left( \frac{r_i(x)}{x^i(x-a)^g} \right) y\\
        & = \frac{f'(x)r_i(x)}{2x^i(x-a)^gy}dx + \left( \frac{r_i'(x)}{x^i(x-a)^g} -\frac{i r_i(x)}{x^{i+1}(x-a)^g} - \frac{gr_i(x)}{x^i(x-a)^{g+1}}\right) ydx \\
        & = \frac{xf'(x)r_i(x)(x-a) + 2f(x)\left(xr_i'(x)(x-a) - ir_i(x)(x-a) - gxr_i(x)\right)}{2x^{i+1}(x-a)^{g+1}y} dx.
        \end{align*}
We now recall that
        \[
        xf'(x) - 2if(x)
        =\psi_i(x) + \phi_i(x).
        \]
We furthermore recall that $r_i(x) = g^{\le i-1}(x)$ where $g(x) = (x-a)^g$. Therefore
\begin{align*}
r'_i(x)&\cdot (x-a) - g \cdot r_i(x)\\
&= \left[g'(x)\right]^{\le i-2}\cdot (x-a) - g \cdot g^{\le i-1}(x)\\
&= \left(\left[g'(x)\cdot (x-a)\right]^{\le i-1} + a\cdot b_{i-1}\cdot x^{i-1} \right) - g \cdot g^{\le i-1}(x)\\
&= a\cdot b_{i-1}\cdot x^{i-1}
\end{align*}
where $b_{i-1} = (-1)^{g-i} i \binom{g}{i} a^{g-i}$ denotes the coefficient of $x^{i-1}$ in $g'(x)$.

Thus we obtain
\begin{align*}
&df_{0 a i} = \frac{(\psi_i(x) + \phi_i(x)) r_i(x)(x-a) +2 f(x) (ab_{i-1} x^i)}{2x^{i+1}(x-a)^{g+1}y} dx\\
&= \frac{\psi_i(x) \left((x-a)^{g+1} - t_i(x)(x-a)\right) +\phi_i(x) r_i(x)(x-a) + 2 f(x) (ab_{i-1} ) x^i}{2x^{i+1}(x-a)^{g+1}y}dx\\
&= \frac{\psi_i(x)}{2x^{i+1}y}dx - \frac{(\psi_i(x)t_i(x) - \phi_i(x)r_i(x))(x-a) - 2 f(x) (ab_{i-1} ) x^i}{2x^{i+1}(x-a)^{g+1}y}dx\\
&= \omega_{0i} -\omega_{ai},
\end{align*}
as claimed.

From the above we moreover obtain that
            \[
        {df_{a \infty i} =} df_{0 \infty i} - df_{0 a i} = (\omega_{0 i} - \omega_{\infty i }) - (\omega_{0 i} - \omega_{a i}) = {\omega_{a i} - \omega_{\infty i}}.
        \]

Finally, $\omega_{ai}$ is regular on $U_a$ because $\omega_{ai}=\omega_{0i}-df_{0ai}$ and $\omega_{ai}$ is hence regular on $U_0 \cap U_a$ and because $\omega_{ai}= \omega_{\infty i} + df_{a \infty i}$ and $\omega_{ai}$ is hence regular on $U_a \cap U_\infty$.
\end{proof}

\section{Actions on $\derhamhone$}

In this section we study the action of certain automorphisms on $\derhamhone$. We first prove that the hyperelliptic involution acts by multiplication by $-1$ on $\derhamhone$ when $p\not= 2$ and as the identity when $p=2$. We then give a family of hyperelliptic curves for which the Hodge-de-Rham short exact sequence \eqref{equationderhamcohomologyshortexactseqeunce} does not split equivariantly.

\begin{thm}\label{thmactionhyperelliptic}
The hyperelliptic involution acts on $\derhamhone$ by multiplication by $-1$.
\end{thm}

\begin{proof}
    Recall that the hyperelliptic involution is the unique non-trivial automorphism $\sigma$ of $X$ such that $ \pi \circ \sigma= \pi$.

    If $p\not= 2$, the  involution $\sigma$ acts on $K(X)$ by $(x,y) \mapsto (x,-y)$.
    Hence, $\sigma$ maps each entry of the triples in (\ref{equationhonebasiselementofderhampnot2}) and (\ref{equationhzerobasiselementofderhampnot2}) to its negative. Thus, the $p \not= 2$ part of Theorem~\ref{theorembasisofderham} implies Theorem~\ref{thmactionhyperelliptic}.

    If $p =2$, the involution $\sigma$ acts on $K(X)$ by $(x,y) \mapsto (x,y+h(x))$. In particular, it fixes the basis elements (\ref{equationhzerobasiselementofderhampis2}) of $\derhamhone$. According to the $p=2$ part of Theorem~\ref{theorembasisofderham}, it remains to show that $\sigma$ also fixes the residue classes $[(\omega_{0i}, \omega_{\infty i}, f_{0 \infty i})]$, $i=1, \ldots, g$, in
    (\ref{equationhonebasiselementofderhampis2}). For $i \in \{1, \ldots, g\}$, this follows from the description of $\derhamhone$ given in (\ref{equationderhamspace}) and (\ref{equationderhamquotient}) and from the equation
    \begin{align*}
    \sigma((\omega_{0i}, \omega_{\infty i}, f_{0 \infty i})) &- (\omega_{0i}, \omega_{\infty i}, f_{0 \infty i}) \\
    &= \left(d \left(\frac{h^{\le i}(x)}{x^i}\right), d\left(\frac{h^{> i}(x)}{x^i}\right), \frac{h^{\le i}(x)}{x^i} -\frac{ h^{> i}(x)}{x^i}\right)
    \end{align*}
    which in turn is verified in the following three lines (where we replace all minus signs with plus signs):
     \[\sigma \left( \frac{\psi_i(x,y)}{x^{i+1}h(x)} dx\right) + \frac{\psi_i(x,y)}{x^{i+1}h(x)}dx  = \frac{[xh'(x) + ih(x)]^{\le i}h(x)}{x^{i+1}h(x)}dx = d\left( \frac{h^{\le i}(x)}{x^i} \right)\]
     \[\sigma \left( \frac{\phi_i(x,y)}{x^{i+1}h(x)} dx\right) + \frac{\phi_i(x,y)}{x^{i+1}h(x)}dx  = \frac{[xh'(x) + ih(x)]^{> i}h(x)}{x^{i+1}h(x)}dx = d\left( \frac{h^{> i}(x)}{x^i} \right)\]
     \[\sigma\left(\frac{y}{x^i}\right) + \frac{y}{x^i} = \frac{h(x)}{x^i}.\]
\end{proof}

\begin{rmk}{\em
If $p\not=2$, Theorem~\ref{thmactionhyperelliptic} can also be proved as follows.
By Lemma~\ref{bais of global holomorphic differentials}, the involution $\sigma$ acts by multiplication by $-1$ on $H^0(X,\Omega_X)$. By Serre duality, it then acts by multiplication by $-1$ also on $H^1(X, \cO_X)$. Finally, by Maschke's Theorem (for the cyclic group of order $2$) applied to the Hodge-de-Rham short exact sequence~\eqref{equationderhamcohomologyshortexactseqeunce}, it acts by multiplication by $-1$ also on $H^1_\mathrm{dR}(X/k)$.}
\end{rmk}

Before we state the main result of this paper, we recall that any automorphism $\tau$ of $X$ induces a map $\bar \tau \colon \mathbb{P}_k^1 \ra \mathbb{P}_k^1$ since $\mathbb {P}_k^1$ is the quotient of~$X$ by the hyperelliptic involution and since the hyperelliptic involution $\sigma$ belongs to the centre of $\aut(X)$ (see \cite[Corollary~7.4.31]{liu}). The following commutative diagram visualises this situation:
    \[
    \xymatrix{
    X \ar[r]^\tau \ar[d]^\pi &  X \ar[d]^\pi \\
    \mathbb P^1_k \ar[r]^{\bar{\tau}}& \mathbb P_k^1
    }  \]

\begin{thm}\label{theoremsplittingtheorem} Let $p \ge 3$. We assume that the degree of the polynomial~$f(x)$ defining the hyperelliptic curve $X$ is odd. We furthermore assume that there exists $\tau \in \aut(X)$ such that the induced map $\bar{\tau}\colon \mathbb{P}_k^1 \ra \mathbb{P}_k^1$ is given by $x \mapsto x+a$ for some $a \not= 0$. Let $G$ denote the subgroup of $\aut(X)$ generated by~$\tau$. Then the Hodge-de-Rham short exact sequence~\eqref{equationderhamcohomologyshortexactseqeunce} does not split as a sequence of $k[G]$-modules.
\end{thm}

The following example explicitly describes hyperelliptic curves that allow an automorphism $\tau$ as assumed in the previous theorem.

\begin{example}\label{examplesofsutomorpohism}
{\em
Let $p \ge 3$, let $a \in k^\times$ and let $q(z) \in k[z]$ be any monic polynomial without repeated roots. Then $f(x) := q(x^p - a^{p-1}x) \in k[x]$ obviously has no repeated roots either and thus $y^2=f(x)$ defines a hyperelliptic curve~$X$. Moreover, $(x,y) \mapsto (x+a,y)$  defines an automorphism~$\tau$ of $X$ and the induced automorphism~$\bar{\tau}$ is given by $x \mapsto x+a$.}
\end{example}

\begin{rmk}\label{RemarksMainTheorem}{\em \hspace{1em}\\
(a) When applied to $a=1$ and to the hyperelliptic curve $X$ given by $q(z) = z$ in Example~\ref{examplesofsutomorpohism}, Theorem~\ref{theoremsplittingtheorem} becomes the main theorem of \cite{canonicalrepresentation}.\\
(b) Theorem~\ref{theoremsplittingtheorem} and Example~\ref{examplesofsutomorpohism} imply that, for every algebraically closed field $k$ of characteristic $p\ge 3$, there exist infinitely many $g \ge 2$ and hyperelliptic curves over~$k$ of genus $g$ for which the Hodge-de-Rham spectral sequence does not split equivariantly.\\
(c) Suppose $g \ge 2$ is given. If $p$ is a prime divisor of $2g+1$ then, according to Theorem~\ref{theoremsplittingtheorem} and Example~\ref{examplesofsutomorpohism}, every monic polynomial $q(z) \in k[z]$ of degree $(2g+1)/p$ without repeated roots defines a hyperelliptic curve $X$ of genus $g$ for which the Hodge-de-Rham sequence does not split equivariantly.
}\end{rmk}

The following proposition shows that any hyperelliptic curve satisfying the assumptions of Theorem~\ref{theoremsplittingtheorem} is in fact of the form as given in Example~\ref{examplesofsutomorpohism}.

\begin{prop}\label{lemmatauactsbyplusminusoneony}
    Let $p \ge 3$ and let $\tau \in \aut(X)$. If the induced isomorphism  $\bar \tau \colon \mathbb P_k^1 \ra \mathbb P_k^1$ is given by $x \mapsto x+ a$ for some $a \not= 0$, then the action of $\tau^*$ on $y$ is given by $\tau^*(y) = y$ or $\tau^*(y) = -y$ and $f(x)$ is of the form $f(x) = q(x^p-a^{p-1} x)$ for some polynomial $q \in k[z]$ without repeated roots.
\end{prop}

\begin{proof}
    We first show that $\tau^*(y) = \pm y$. There exist $g_1(x)$ and $g_2(x)\not= 0 $ in $k(x)$ such that
        \begin{equation*}
        \tau^*(y) = g_1(x) + g_2(x)y \in k(x,y).
        \end{equation*}
    Hence
        \begin{equation}\label{equationexpandingysquared}
        f(x+a) = \tau^*(y^2) = (\tau^*(y))^2 = g_1(x)^2 + 2g_1(x)g_2(x)y + g_2(x)^2f(x).
        \end{equation}
   This implies that $g_1(x) =0$ because otherwise
        \[
        y = \frac{f(x+a) - g_1(x)^2 - g_2(x)^2f(x)}{2g_2(x)g_1(x)}
        \]
   would belong to $k(x)$. By comparing the degrees in \eqref{equationexpandingysquared} we see that $g_2(x)$ is a constant, and then by comparing coefficients in the same equation we see that $g_2(x)^2 = 1$.
    Hence $\tau^*(y) = \pm y$, as claimed.

    We now show that $f(x)$ is of the form $q(x^p-a^{p-1}x)$. The extension $k(x)= k(z,x)$ of the rational function field $k(z)$ obtained by adjoining an element $x$ satisfying the equation $x^p-a^{p-1}x-z=0$ is a Galois extension with cyclic Galois group generated by the automorphism $x \mapsto x+a$. We derived above that $f(x) = f(x+a)$. Hence $f(x)$ belongs to $k(z)$. Furthermore, $x$ and hence $f(x)$ is integral over $k[z]$. Therefore, $f(x)\in k(z)$ belongs to $k[z]$, i.e., $f(x) = q(x^p-a^{p-1}x)$ for some $q \in k[z]$ without repeated roots, as was to be shown.
    \end{proof}

\begin{proof}[Proof (of Theorem~\ref{theoremsplittingtheorem})]
We suppose that the sequence \eqref{equationderhamcohomologyshortexactseqeunce} does split and that
\[s \colon \hone \ra \derhamhone\]
is a $k[G]$-linear splitting map.
    Then we have
        \begin{equation}\label{equationcommutativityoftauandsplittingmap}
        s(\tau^*(\alpha)) = \tau^*(s(\alpha)) \in \derhamhone
        \end{equation}
    and
        \begin{equation}p(s(\alpha)) = \alpha\end{equation}
    for all $\alpha \in \hone$.
    We will show that these equalities give rise to a contradiction when $\alpha$ is the residue class $ \left[ \frac{y}{x^g}\right]$ in $\hone$ (see Proposition~\ref{theorembasisofhone}).

We first show that $\left[\frac{y}{x^g}\right] \in \hone$ is fixed by $\tau^*$.
   To this end, we recall that $\cU=\{U_0, U_\infty\}$, $\cU'= \{U_a, U_\infty\}$ and $\cU''=\{U_0, U_a, U_\infty\}$ and consider the following obviously commutative diagram of isomorphisms where $\rho$ and $\rho'$ are defined as in \eqref{equationdefinitionofrho} and the equalities denote the identification~(\ref{equationcechhoneisomorphismfunctions}):
        \[\xymatrix{
        \hone \ar@{=}[r] \ar[d]^{\tau^*}& \check{H}^1(\cU,\cO_X)  &  \check{H}^1(\cU'',\cO_X)  \ar[l]_{\rho} \ar[d]^{\rho'}\\
        \hone \ar@{=}[r] & \check{H}^1(\cU,\cO_X)  & \check{H}^1(\cU',\cO_X) \ar[l]_{\tau^*}.
        } \]
    By (the proof of) Proposition~\ref{propbasisoftriplecoverderham}, the triple $\left( \frac{r_g(x)y}{x^g(x-a)^g}, \frac{y}{x^g}, \frac{t_g(x)y}{x^g(x-a)^g} \right)$ defines a well-defined element of $\check{H}^1(\cU'',\cO_X)$. Hence we have
          \begin{align*}
    \rho^{-1} \left(\left[\frac{y}{x^g}\right]\right) &=
        \left[ \left( \frac{r_g(x)y}{x^g(x-a)^g}, \frac{y}{x^g}, \frac{t_g(x)y}{x^g(x-a)^g} \right)\right] \\
        &= \left[ \left( \frac{((x-a)^g-x^g)y}{x^g(x-a)^g}, \frac{y}{x^g}, \frac{y}{(x-a)^g} \right) \right] \textrm{ in } \check{H}^1(\cU'',\cO_X).
        \end{align*}
    We therefore obtain
        \begin{align*}
        \tau^* \left( \left[\frac{y}{x^g}\right]\right)  & = \tau^* \left( \rho' \left(\rho^{-1}\left(\left[\frac{y}{x^g}\right]\right)\right)\right) \\
        & = \tau^*\left( \rho' \left( \left[ \left( \frac{((x-a)^g-x^g)y}{x^g(x-a)^g}, \frac{y}{x^g}, \frac{y}{(x-a)^g} \right) \right] \right) \right) \\
        & = \tau^*\left( \left[ \frac{y}{(x-a)^g} \right] \right) = \left[ \frac{y}{x^g} \right],
        \end{align*}
    as claimed.

    By Theorem~\ref{theorembasisofderham}, the elements $\lambda_i$, $i=0, \ldots, g-1$, defined in \eqref{equationhzerobasiselementofderhampnot2} together with the elements $\gamma_i$, $i=1, \ldots, g$, defined in \eqref{equationhonebasiselementofderhampnot2} form a basis of $\derhamhone$.
        Since the canonical projection $p: \derhamhone \ra \hone$ is $k[G]$-linear and maps $\gamma_g$ to the residue class $\left[ \frac{y}{x^g} \right]$, it follows that
        \begin{equation}\label{equation defining lambd_i}
        \tau^*(\gamma_g) = \gamma_g + \sum_{i=0}^{g-1} c_i\lambda_i
        \end{equation}
    for some $c_0, \ldots, c_{g-1} \in k$.
    On the other hand, we have
        \[
        s\left( \left[ \frac{y}{x^g} \right] \right)  = \gamma_g + \sum_{i=0}^{g-1}d_i \lambda_i
        \]
    for some $d_0, \ldots, d_{g-1} \in k$.
    Now the action of $\tau^*$ on $\lambda_i$ for $0 \leq i \leq g-1$ is easily seen to be given by
        \begin{equation}\label{lemma action on lambda_i}
        \begin{split}
        \tau^*(\lambda_i) & = \tau^*\left( \left[ \left( \frac{x^i}{y}dx, \frac{x^i}{y}dx, 0\right) \right] \right) \\
        & = \left[ \left( \frac{(x+a)^i}{y}dx, \frac{(x+a)^i}{y}dx, 0 \right) \right] = \sum_{k=0}^i \binom{i}{k} a^{i-k} \lambda_k.
        \end{split}
        \end{equation}
    Plugging the equations obtained so far into equation~\eqref{equationcommutativityoftauandsplittingmap} we obtain
        \begin{align*}
        \gamma_g + \sum_{i=0}^{g-1} d_i\lambda_i & = s \left( \left[ \frac{y}{x^g} \right] \right)
        = s \left( \tau^*\left( \left[ \frac{y}{x^g} \right] \right) \right) \\
        & = \tau^* \left(s \left( \left[ \frac{y}{x^g} \right] \right) \right) = \tau^* \left( \gamma_g + \sum_{i=0}^{g-1} d_i \lambda_i \right) \\
        & = \left( \gamma_g + \sum_{i=0}^{g-1} c_i \lambda_i \right) + \sum_{i=0}^{g-1} d_i \left( \sum_{k = 0}^i \binom{i}{k} a^{i-k} \lambda_k \right).
        \end{align*}
    By comparing coefficients of the basis element $\lambda_{g-1}$, we see that $c_{g-1} = 0$.
   On the other hand, we will below derive the equation $c_{g-1} = a/4$ from the defining equation~\eqref{equation defining lambd_i}. Since we assumed that $a \neq 0$, this gives us the desired contradiction.

   The left-hand side of equation~\eqref{equation defining lambd_i} is $\tau^*(\gamma_g)$. To compute $\tau^*(\gamma_g)$ we consider the following commutative diagram of isomorphisms
        where $\rho$ is the canonical projection \eqref{equationdefinitionofrho}, $\rho'$ is given by
            $
             (\omega_0, \omega_a, \omega_\infty, f_{0 a}, f_{0 \infty}, f_{a \infty}) \mapsto (\omega_a, \omega_\infty, f_{a \infty})
            $
        and the equalities denote the identification given by~\eqref{equationderhamspace} and~\eqref{equationderhamquotient}:
        \begin{equation}\label{equationderhamcommutativediagram}
        \begin{split}
        \xymatrix{
        \derhamhone \ar@{=}[r] \ar[d]^{\tau^*}& \check{H}^1_\mathrm{dR}(\cU)   & \check{H}^1_\mathrm{dR}(\cU'') \ar[l]_{\rho} \ar[d]^{\rho'} \\
        \derhamhone \ar@{=}[r] & \check{H}^1_\mathrm{dR}(\cU)  & \check{H}^1_\mathrm{dR}(\cU') \ar[l]_{\tau^*}.
        }
        \end{split}
        \end{equation}
    Then, by Proposition \ref{propbasisoftriplecoverderham}, we have:
        \begin{equation}\label{equationactionoftauongamma}
        \begin{split}
        \tau^*(\gamma_g) & = \tau^*(\rho'(\rho^{-1}(\gamma_g))) \\
        & = \tau^*\left( \left[ \omega_{a g}, \frac{- \phi_g(x)}{2x^{g+1}y} dx, \frac{y}{(x-a)^g} \right) \right] \\
        & = \left[ \left( \tau^*(\omega_{a g}) , \frac{-\phi_g(x+a)}{2(x+a)^{g+1}y}dx, \frac{y}{x^g} \right) \right].
        \end{split}
        \end{equation}
    On the other hand, the right hand side of equation \eqref{equation defining lambd_i} is equal to
        \begin{equation}\label{equationimageofgammaunders}
        \left[ \left( \frac{\psi_g(x)}{2x^{g+1}y}dx, \frac{-\phi_g(x)}{2x^{g+1}y}dx, \frac{y}{x^g} \right) \right] \\ + \sum_{i=0}^{g-1} c_i \left[ \left( \frac{x^i}{y}dx, \frac{x^i}{y}dx, 0 \right) \right].
        \end{equation}
    Note that the third entry in both \eqref{equationactionoftauongamma} and \eqref{equationimageofgammaunders} is $\frac{y}{x^g}$.
    Now, if the third entry of a triple $(df_0, df_\infty, f_0|_{U_0 \cap U_\infty}- f_\infty|_{U_0\cap U_\infty})$ in the subspace \eqref{equationderhamquotient} of the space~\eqref{equationderhamspace} vanishes, then $f_0$ and $f_\infty$ glue to a global and hence constant function and the whole triple vanishes. Hence, the triples in  \eqref{equationactionoftauongamma} and \eqref{equationimageofgammaunders} are  equal already before taking residue classes. By comparing the second entries of \eqref{equationactionoftauongamma} and \eqref{equationimageofgammaunders} we therefore obtain the equation
        \[
        - \frac{\phi_g(x+a)}{2(x+a)^{g+1}y} dx = -\frac{\phi_g(x)}{2x^{g+1}y}dx + \sum_{i = 0}^{g-1} c_i \frac{x^i}{y}dx \quad \textrm{ in } \quad \Omega_{K(X)}.
        \]
    Since $dx$ is a basis of $\Omega_{K(X)}$ considered as a $K(X)$-vector space, the equation above is equivalent to the equation
        \[
        \frac{\phi_g(x+a)}{2(x+a)^{g+1}} = \frac{\phi_g(x)}{2x^{g+1}} - \sum_{i=0}^{g-1} c_i x^i
        \]
(in the rational function field $k(x)$) which in turn is equivalent to the equation
        \[
        \phi_g(x+a)x^{g+1} = \phi_g(x)(x+a)^{g+1} - 2(x+a)^{g+1}x^{g+1}\sum_{i=0}^{g-1}c_i x^i \quad \textrm{ in } \quad k[x].
        \]
    Now, the assumption that the degree of $f(x)$ is odd  means that the degree of~$f(x)$ is precisely $2g+1$. By definition, the terms of highest degree in $\phi_g(x)$ are the same as the terms of highest degree in
        \[
        s_g(x) = xf'(x) - 2gf(x) = x^{2g+1} + 0\cdot x^{2g} + \ldots.
        \]
    We therefore have
        \begin{multline*}
        \left( (x+a)^{2g+1} + 0 \cdot (x+a)^{2g} + \ldots \right) x^{g+1}  \\ = (x^{2g+1} + 0 \cdot x^{2g} + \ldots )(x+a)^{g+1} - 2(x+a)^{g+1}x^{g+1}(c_{g-1}x^{g-1} + \ldots ).
        \end{multline*}
    Hence, by comparing the coefficients of $x^{3g+1}$, we obtain
        \[
        (2g+1)a = (g+1)a - 2c_{g-1}.
        \]
    and hence
    \[c_{g-1} = \frac{((g+1) - (2g+1))a}{2} = - \frac{g}{2} a.\]
    Finally, we have $2g + 1= \deg(f(x)) \equiv 0 \mod \; p$ by Proposition~\ref{lemmatauactsbyplusminusoneony} and hence
        \[
        c_{g-1} = \frac{a}{4},
        \]
    as claimed above.
    This concludes the proof of Theorem \ref{theoremsplittingtheorem}.
   \end{proof}

\begin{rmk}\label{mistake2}
{\em While the method of calculating $\tau^*\left(\left[\frac{y}{x^g}\right]\right)$ and $\tau^*(\gamma_g)$ in the proof above is the same as in \cite{canonicalrepresentation}, the actual computations do not generalise those in \cite{canonicalrepresentation}, not only due to Remark~\ref{mistake} but also because of further mistakes in \cite{canonicalrepresentation}. Finally, the argument in the proof above for obtaining the desired contradiction is different from the one in \cite{canonicalrepresentation} the very end of which has unfortunately not been carried out anyway.}
\end{rmk}

We conclude with an example which demonstrates that the requirement in Theorem \ref{theoremsplittingtheorem} of $f(x)$ to be of odd degree is a necessary condition.

\begin{example}\label{ExampleBranchPoint}
{\em
    Let $p = 3$ and $X$ be the hyperelliptic curve of genus $2$ defined by the equation
        \[
        y^2 = f(x) =  x^6 + x^4 + x^2 + 2.
        \]

     As in Theorem~\ref{theoremsplittingtheorem} let $\tau$ denote the automorphism of $X$ given by $(x,y) \mapsto (x+1,y)$ and let $G:= \langle \tau \rangle$.

    By Theorem \ref{theorembasisofderham}, a basis of $\check{H}^1_\mathrm{dR}(\cU)$ is given by
        \[
        \lambda_0  =\left[ \left( \frac{1}{y}dx, \frac{1}{y}dx, 0 \right) \right], \quad
        \lambda_1  =\left[ \left(\frac{x}{y}dx, \frac{x}{y}dx, 0\right) \right], \]
        \[\gamma_1  =\left[\left( \frac{1}{x^2y}dx, \frac{x^4+2x^2}{y}dx, \frac{y}{x} \right) \right], \quad
        \gamma_2  =\left[\left( \frac{x^2 +1}{2x^3y}dx, \frac{2x^3}{y}dx, \frac{y}{x^2}\right) \right].
        \]
    By Proposition~\ref{theorembasisofhone}, the residue classes $\bar{\gamma}_1:= \left[\frac{y}{x}\right]$ and $\bar \gamma_2:= \left[\frac{y}{x^2}\right]$ form a basis of $\hone$.
    We define a map
        \[
        s \colon \hone \ra \derhamhone
        \]
        of vector spaces over $k$ by
        \[
        \bar \gamma_1 \mapsto \gamma_1 \quad \text{and} \quad \bar \gamma_2 \mapsto \gamma_2 + \lambda_1.
        \]
    Clearly $p \circ s$ is the identity map on $\hone$, and hence, if $s$ is $k[G]$-linear, the sequence in Proposition \ref{propshortexactsequence} does split as a sequence of $k[G]$-modules.

    We now show that $s$ is $k[G]$-linear.
    By Proposition \ref{propbasisoftriplecoverderham}, the pre-images of~$\gamma_1$ and $\gamma_2$ under $\rho$ in $\check{H}^1_\mathrm{dR}(\cU'')$ are the residue classes of
        \[
        \nu_1 = \left(\frac{1}{x^2y}dx, \frac{x^4+2x^3+2x^2}{2(x-1)^3y}dx, \frac{x^4+2x^2}{y}dx, \frac{y}{x(x-1)^2}, \frac{y}{x}, \frac{(x+1)y}{(x-1)^2} \right)
        \]
    and
        \[
        \nu_2 = \left( \frac{x^2+1}{2x^3y}dx, \frac{x^3+x^2+x+1}{2(x-1)^3y}dx, \frac{2x^3}{y}dx, \frac{(x+1)y}{x^2(x-1)^2}, \frac{y}{x^2}, \frac{y}{(x-1)^2} \right),
        \]
    respectively. Using a computation similar to \eqref{equationactionoftauongamma}, it is easy to verify that
        \begin{align*}
        \tau^*(\gamma_1) &= \tau^* (\rho'(\nu_1)) \\
        &=\left[\left(\frac{x^4+2x^2 +2x+2}{2x^3y}dx, \frac{x^4+x^3+2x^2+2x}{y}dx, \frac{(x+2)y}{x^2}\right)\right]\\
        &= \gamma_1 + 2\gamma_2 + 2\lambda_1
        \end{align*}
    and that
        \begin{align*}
        \tau^*(\gamma_2) &= \tau^*(\rho'(\nu_2)) \\
        &=\left[\left(\frac{x^3+x^2+1}{2x^3y}dx, \frac{2x^3+2}{y}dx, \frac{y}{x^2}\right)\right]\\
        &= \gamma_2 + 2 \lambda_0.
        \end{align*}
    Furthermore, we have seen in \eqref{lemma action on lambda_i} that
        \[
        \tau^*(\lambda_0) = \lambda_0 \quad \text{and} \quad \tau^*(\lambda_1) = \lambda_1 + \lambda_0.
        \]
    We finally conclude that
        \[
        s(\tau^*(\bar\gamma_1)) = s(\bar\gamma_1 + 2\bar \gamma_2) = \gamma_1 + 2\gamma_2 + 2\lambda_1 = \tau^*(\gamma_1) = \tau^*(s(\bar\gamma_1))
        \]
    and
        \[
        s(\tau^*(\bar\gamma_2)) = s(\bar\gamma_2) = \gamma_2 + \lambda_1 = \tau^*(\gamma_2 + \lambda_1)  = \tau^*(s(\bar\gamma_2)).
        \]
    Hence $s$ is $k[G]$-linear, and the Hodge-de-Rham short exact sequence \eqref{equationderhamcohomologyshortexactseqeunce} splits.}
    \end{example}

\begin{rmk}\label{RemarkBranchPoint}
{\em
    The curve $X$ defined in the previous example is isomorphic to the modular curve $X_0(22)$.

    To see this, we first note that, by \cite[Table 2]{automorphismshyperellipticmodular}, the modular curve $X_0(22)$ is the hyperelliptic curve of genus $2$ defined by
        \[
        y^2 = f(x) = x^6 + 2x^4 +x^3 +2x^2 +1.
        \]
    Now,  $x \mapsto x-1$, $y \mapsto y$ defines an isomorphism between $X_0(22)$ and the curve defined by $y^2 = x^6 + 2x^4 + 2x^2 + 2$.
    We finally apply the isomorphism described in the following general procedure.

    If $g(x) = a_sx^s + \ldots + a_0$ with $a_0 \neq 0 \neq a_s$, we define $g^*(x) := a_0^{-1}x^sg\left( \frac{1}{x}\right)$.
    It is stated after Lemma 2.6 in \cite{automorphismshyperellipticmodular} that, if $y^2 = g(x)$ defines a hyperelliptic curve and $s$ is even, then the curves defined by $y^2=g(x)$ and $y^2 = g^*(x)$ are isomorphic.}
  \end{rmk}

\end{document}